\documentclass{amsart}
\usepackage[english]{babel}
\usepackage{amssymb}
\usepackage{hyperref}   
\usepackage{mathtools}
\mathtoolsset{showonlyrefs}

\theoremstyle{plain}
\newtheorem{definition}{Definition}
\theoremstyle{plain}
\newtheorem{lemma}{Lemma}
\theoremstyle{remark}
\newtheorem{remark}{Remark}
\theoremstyle{plain}
\newtheorem{theorem}{Theorem}

\begin{document}

\title{asymptotically hyperbolic manifold with a horospherical boundary}

\author{Xiaoxiang Chai}
\address{Korea Institute for Advanced Study, Seoul 02455, South Korea}
\email{xxchai@kias.re.kr}

\begin{abstract}
  We discuss asymptotically hyperbolic manifold with a noncompact boundary
  which is close to a horosphere in a certain sense. The model case is a
  horoball or the complement of a horoball in standard hyperbolic space. We
  show some geometric formulas.
\end{abstract}

{\maketitle}

\section{Introduction}

In the upper half-space model
\begin{equation}
  b = \tfrac{1}{(x^n)^2} ((\mathrm{d} x^1)^2 + \cdots + (\mathrm{d} x^n)^2),
  x^n > 0
\end{equation}
of hyperbolic $n$-space $\mathbb{H}^n$, we fix the horosphere $\mathcal{H}=
\{x \in \mathbb{H}^n :^{} x^n = 1\}$.

\begin{definition}
  We say that a manifold $(M^n, g)$ with a noncompact boundary is
  asymptotically hyperbolic with a horospherical{\upshape{}} boundary if
  outside a compact set $K \subset M$, $M$ is diffeomorphic to $\{x^n
  \leqslant 1\}$ ($\{x^n \geqslant 1\}$) minus a compact set and the metric
  admits the decay rate
  \begin{equation}
    |e|_b + | \bar{\nabla} e|_b + | \bar{\nabla} \bar{\nabla} e|_b = O
    (\mathrm{e}^{- \tau r}) \label{decay rate}
  \end{equation}
  where $e = g - b$, $\tau > \tfrac{n}{2}$ and $r$ is the $b$-geodesic
  distance to a fixed point $o$ (in $\mathbb{H}^n$).
\end{definition}

We pick $o$ to be $(0, \ldots, 1)$ without loss of generality, the
$b$-distance to $o$ is then
\begin{equation}
  2 \cosh r = \tfrac{1}{x^n} (| \hat{x} |^2 + (x^n)^2 + 1), \label{distance
  formula}
\end{equation}
where $\hat{x} = (x^1, \ldots, x^{n - 1})$ and $| \hat{x} |$ is the Euclidean
distance of $\hat{x}$ to the origin $\hat{o}$. See for example {\cite[Chapter
A]{benedetti-lectures-1992}} for the distance formula. Replacing
$\mathrm{e}^r$ with $\cosh r$ in {\eqref{decay rate}} is somehow more
convenient. We are concerned with behaviors near infinity, we assume in this
article that the manifold $M$ is diffeomorphic to $\{x^n \leqslant 1\}$ with a
smooth metric $g$. The other case $\{x^n \geqslant 1\}$ is handled in the same
way, for clarity of presentation, we omit this case and leave the details to
the reader.

For $\varepsilon \in (0, 1)$, we define
\[ C_{\varepsilon} = \{\varepsilon \leqslant x^n \leqslant 1\} \cap \{|
   \hat{x} | \leqslant \rho (\varepsilon)\} . \]
We also write $C (\varepsilon)$ sometimes for clarity and also for others
defined below. We denote $A_{\varepsilon_1, \varepsilon_2} = C_{\varepsilon_2}
\backslash C_{\varepsilon_1}$ with $\varepsilon_2 < \varepsilon_1$. Inner
boundary $\partial' C_{\varepsilon} = \partial C_{\varepsilon}
\backslash\mathcal{H}$ of $C_{\varepsilon}$ is the union of
\[ F_{\varepsilon} = \{| \hat{x} | \leqslant \rho (\varepsilon), x^n =
   \varepsilon\} \]
and
\begin{equation}
  S_{\varepsilon} = \{\varepsilon \leqslant x^n \leqslant 1, | \hat{x} | =
  \rho (\varepsilon)\} .
\end{equation}
We deonte by $\nu$ the $g$-normal to $\partial' C_{\varepsilon}$.

Along the horosphere $\mathcal{H}$,
\begin{equation}
  c_{\varepsilon} = \{x^n = 1, | \hat{x} | \leqslant \rho (\varepsilon)\}
\end{equation}
and $a_{\varepsilon_2, \varepsilon_1} = c_{\varepsilon_2} \backslash
c_{\varepsilon_1}$, and $s_{\varepsilon} = \{x^n = 1, | \hat{x} | = \rho
(\varepsilon)\}$. We denote by $\mu$ the normal to $s_{\varepsilon}$ in the
horosphere.

Here $\rho (\varepsilon)$ is a smooth decreasing function on $(0, 1)$
satisfying $\rho (\varepsilon) \to \infty$ as $\varepsilon \to 0$. Note that
$C_{\varepsilon}$ is a parabolic cylinder in the region $\{x^n \leqslant 1\}$.
Analogously, we can define
\[ \mathcal{C}_{\varepsilon} = \{1 \leqslant x^n \leqslant \varepsilon^{- 1}
   \} \cap \{| \hat{x} | \leqslant \rho (\varepsilon)\} . \]
We also write $C (\varepsilon)$ sometimes for clarity. We denote
$\mathcal{A}_{\varepsilon_1, \varepsilon_2} =\mathcal{C}_{\varepsilon_2}
\backslash\mathcal{C}_{\varepsilon_1}$ with $\varepsilon_2 < \varepsilon_1$.
Inner boundary $\partial' \mathcal{C}_{\varepsilon} = \partial
\mathcal{C}_{\varepsilon} \backslash\mathcal{H}$ of
$\mathcal{C}_{\varepsilon}$ is the union of
\[ \mathcal{F}_{\varepsilon} = \{| \hat{x} | \leqslant \rho (\varepsilon), x^n
   = \varepsilon^{- 1} \} \]
and
\begin{equation}
  \mathcal{S}_{\varepsilon} = \{1 \leqslant x^n \leqslant \varepsilon^{- 1}, |
  \hat{x} | = \rho (\varepsilon)\} .
\end{equation}
Let $V_0 = \tfrac{1}{x^n}$ and $V_k = \tfrac{x^k}{x^n}$ with $k$ ranges from 1
to $n - 1$. We will show later that $V_0$ and $V_k$ is a function satisfying
the following condition. Let $\bar{\eta} = \partial_n$, we see that
\begin{equation}
  \bar{\nabla}^2 V = V b, \partial_{\bar{\eta}} V = - V \text{ along }
  \mathcal{H}. \label{static potential full}
\end{equation}
The functions satisfying the condition $\bar{\nabla}^2 V = V b$ is a
{\itshape{static potential}} which is closely related to the static spacetime.
The scalar curvature $R_g$ admits the decay rate
\begin{equation}
  V (R_g + n (n - 1)) = \bar{\nabla}_i \mathbb{U}^i + O (\mathrm{e}^{- 2 \tau
  r + r}) \label{scalar curvature divergence}
\end{equation}
where
\[ \mathbb{U}= V\ensuremath{\operatorname{div}}e - V \mathrm{d}
   (\ensuremath{\operatorname{tr}}_b e) +\ensuremath{\operatorname{tr}}_b e
   \mathrm{d} V - e (\bar{\nabla} V, \cdot) \]
is the mass integrands. See {\cite{chrusciel-mass-2003}}.

We denote by $\mathrm{d} v$, $\mathrm{d} \sigma$ and $\mathrm{d} \lambda$
respectively the $n$, $n - 1$ and $n - 2$ dimensional volume element. We have
the following definition of mass like quantities $\mathbf{M} (V)$
\begin{equation}
  \mathbf{M} (V) = \lim_{\varepsilon \to 0} \left[ \int_{\partial'
  C_{\varepsilon}} \mathbb{U}^i \bar{\eta}_i \mathrm{d} v - \int_{\partial
  c_{\varepsilon}} V e_{\alpha n} \mu^{\alpha} \mathrm{d} \lambda \right]
  \label{mass}
\end{equation}
whenever exists. This is motivated by {\cite{wang-mass-2001}},
{\cite{chrusciel-mass-2003}}, {\cite{almaraz-mass-2020}} and
{\cite{chai-mass-2018}}. We will show the existence of $\mathbf{M} (V)$ in
Theorem \ref{finiteness}.

Although not explicitly, {\cite{andersson-rigidity-2008}} show that

\begin{theorem}
  \label{metric nonexistence}On $\mathbb{T}^{n - 1} \times [0, 1]$, there does
  not exist a metric $g$ with $R_g \geqslant - n (n - 1)$, $H_g \geqslant - (n
  - 1)$ on the top face $\mathbb{T}^{n - 1} \times \{1\}$ and $H_g \geqslant n
  - 1$ on the bottom face.
\end{theorem}

However, it seems that a spinorial proof was not available explicitly in the
literature. One could use the boundary conditions
{\cite[(4.25)]{chrusciel-mass-2003}} to write down a proof. We see that the
quantities {\eqref{mass}} we defined is a natural candidate to show a similar
positive mass theorem as in
{\cite{wang-mass-2001,chrusciel-mass-2003,almaraz-mass-2020}} thus giving a
noncompact version of Theorem \ref{metric nonexistence}. Although we have not
shown the geometric invariance of $\mathbf{M}$, that is the independence of
$\mathbf{M}$ on the coordinate chart at infinity. The natural conditions
should be $R_g \geqslant - n (n - 1)$ and $H_g + n - 1 \geqslant 0$ along the
horospherical boundary. Similar conjecture can be stated for the horoball. It
is also possible to define an asymptotically hyperbolic manifold with a
noncompact boundary which is an equidistant hypersurface and related
quantities. Another interesting direction is to explore the spacetime version
of $\mathbf{M}$ in {\eqref{mass}}. We will address these questions in a later
work.

We show that the $\mathbf{M} (V)$ can be evaluated via the Ricci tensor and
the second fundamental form the horosphere $\mathcal{H}$ similar to those in
{\cite{chai-mass-2018}}.

\begin{theorem}
  \label{evaluation}We assume that $\rho (\varepsilon) = \varepsilon^{-
  \alpha}$ where $\alpha > \tfrac{4}{3}$.We have that
  \begin{equation}
    \mathbf{M} (V) = \tfrac{2}{2 - n} [\int_{C_{\varepsilon}} G (X, \nu)
    \mathrm{d} v + \int_{\partial C_{\varepsilon}} W (X, \mu) \mathrm{d}
    \lambda] + o (1)
  \end{equation}
  where $G =\ensuremath{\operatorname{Ric}}- \tfrac{1}{2} R g - (n - 1) (n -
  2) g$ is the modified Einstein tensor and $W = A - H g - (n - 2) g$.
\end{theorem}

Here the pair $V$ and $X$ is given in {\eqref{vector and potential}}. The
important property of $X$ is that it is a conformal Killing vector and tangent
to the horosphere $\mathcal{H}$.

The article is organized as follows:

In Section \ref{asymptotics}, we show some asymptotics which motivates the
Definition \ref{mass}. we collect basics of conformal Killing vectors in
standard hyperbolic space and prove some length estimates under the metric
$g$. In Section \ref{proof}, we give a proof of Theorem \ref{evaluation}.

\

{\bfseries{Acknowledgment}} I would like to thank Prof. Kim Inkang for
discussions on tetrahedron in hyperbolic 3-space. I would like to acknowledge
the support of Korea Institute for Advanced study under the research number
MG074401.

\

\section{Background}\label{asymptotics}

The Christoffel symbols of the standard hyperbolic metric $b$ is given by
\begin{equation}
  \bar{\Gamma}_{i n}^j = - (x^n)^{- 1} \delta_i^j, \bar{\Gamma}_{\alpha
  \beta}^{\gamma} = 0, \bar{\Gamma}_{\alpha \beta}^n = (x^n)^{- 1}
  \delta_{\alpha \beta} \label{christoffel background}
\end{equation}
for $i$, $j$ possibly being $n$ ranging from 1 to $n$ and $\alpha, \beta,
\gamma$ ranges from 2 to $n$. We use this convention later as well.

Consider again the half space model and the slice $\{x_n = 1\}$, then (with
all quantities evaluated on the slice $\{x_n = 1\}$)
\begin{equation}
  \Gamma_{\alpha \beta}^n = \delta_{\alpha \beta} + \tfrac{1}{2}
  (\bar{\nabla}_{\alpha} e_{\beta n} + \bar{\nabla}_{\beta} e_{\alpha n} -
  \bar{\nabla}_n e_{\alpha \beta}) . \label{connection diff}
\end{equation}
The outward normal is $\eta^i = (g^{n n})^{- \tfrac{1}{2}} g^{n i}$ and the
second fundamental form is
\begin{equation}
  A_{\alpha \beta} = - (g^{n n})^{- \tfrac{1}{2}} \Gamma_{\alpha \beta}^n .
  \label{second fundamental form}
\end{equation}
Hence the mean curvature is
\begin{equation}
  H = - h^{\alpha \beta} (g^{n n})^{- \tfrac{1}{2}} \Gamma_{\alpha \beta}^n .
  \label{mean curvature}
\end{equation}
Note that $h^{\alpha \beta} = \delta^{\alpha \beta} - e_{\alpha \beta} + O
(\mathrm{e}^{- 2 \tau r})$ and $(g^{n n})^{- \tfrac{1}{2}} = 1 + \tfrac{1}{2}
e_{n n} + O (\mathrm{e}^{- 2 \tau r})$, we have
\begin{align}
  & 2 (H + n - 1) \\
  = & - 2 h^{\alpha \beta} (g^{n n})^{- \tfrac{1}{2}} \Gamma_{\alpha \beta}^n
  \\
  = & - \sum_{\alpha} (2 \bar{\nabla}_{\alpha} e_{\alpha n} - \bar{\nabla}_n
  e_{\alpha \alpha}) - (n - 1) e_{n n} + 2 e_{\alpha \alpha} + O
  (\mathrm{e}^{- 2 \tau r}) . 
\end{align}
Write $\mathbb{U}$ in coordinates, we have with $E
=\ensuremath{\operatorname{tr}}_b e$
\[ \mathbb{U}^i = V g^{i k} \bar{\nabla}^j e_{j k} - V \bar{\nabla}^i E + E
   \bar{\nabla}^i V - g^{i k} e_{j k} \bar{\nabla}^k V. \]
Along $\mathcal{H}$,
\begin{align}
  & \mathbb{U}^i \bar{\eta}_i = \\
  = & V \bar{\eta}^i \bar{\nabla}^j e_{j i} - V \bar{\eta}^i \bar{\nabla}_i
  e_{j j} + E \bar{\eta}^i \bar{\nabla}_i V - \bar{\eta}^i e_{j i}
  \bar{\nabla}^j V \\
  = & V \bar{\nabla}_{\alpha} e_{\alpha n} - e_{\alpha n}
  \bar{\nabla}_{\alpha} V - V \bar{\nabla}_n e_{\alpha \alpha} + e_{\alpha
  \alpha} \bar{\nabla}_n V + O (\mathrm{e}^{- 2 \tau r + r}) . 
\end{align}
So we have that
\begin{align}
  & 2 V (H + n - 1) +\mathbb{U}^i \bar{\eta}_i \\
  = & - V \bar{\nabla}_{\alpha} e_{\alpha n} - e_{\alpha n}
  \bar{\nabla}_{\alpha} V - (n - 1) V e_{n n} + (2 V + \bar{\nabla}_n V)
  e_{\alpha \alpha} + O (\mathrm{e}^{- 2 \tau r + r}) \\
  = & - \bar{\nabla}_{\alpha} (V e_{\alpha n}) - (n - 1) V e_{n n} + (2 V +
  \bar{\nabla}_n V) e_{\alpha \alpha} + O (\mathrm{e}^{- 2 \tau r + r})
  \\
  = & - \partial_{\alpha} (V e_{\alpha n}) + V \bar{\Gamma}_{\alpha n}^i
  e_{\alpha i} + V \bar{\Gamma}_{\alpha \alpha}^i e_{i n} - V (n - 1) e_{n n}
  \\
  & + (2 V + \bar{\nabla}_n V) e_{\alpha \alpha} + O (\mathrm{e}^{- 2 \tau r
  + r}) \\
  = & - \partial_{\alpha} (V e_{\alpha n}) + O (\mathrm{e}^{- 2 \tau r + r}) .
  \label{divergence horosphere} 
\end{align}
\subsection{Conformal Killing vectors}

Since the metric of the hyperbolic space using the upper half space model is
conformal to standard metric of the Euclidean metric. We investigate
$\bar{\nabla}_i X^j$ for such $X$ which is a conformal Killing vector with
respect to the Euclidean metric. We use repeatedly that the fact
\[ \bar{\nabla}_i X^j = \bar{\nabla}^i X_j \]
which follows easily from conformality to $\delta$.

The $X = x^i \partial_i X = x^i \partial_i$, then
\begin{align}
  & \bar{\nabla}_i X^j \\
  = & \partial_i x^j + x^k \Gamma_{i k}^j \\
  = & \delta_i^j + x^k \tfrac{1}{2} (x^n)^2 (g_{i j, k} + g_{k j, i} - g_{i k,
  j}) \\
  = & \delta_i^j + \tfrac{1}{2} (x^n)^2 (x^n \tfrac{\partial}{\partial x^n}
  ((x^n)^{- 2}) \delta_{i j} + x^j g_{j j, i} - x^i g_{i i, j}) \\
  = & \tfrac{1}{2} (x^n)^2 (x^j g_{j j, i} - x^i g_{i i, j}) . 
\end{align}
So $\bar{\nabla}_i X^j + \bar{\nabla}^j X_i = 0$. We see that vector $x^i
\partial_i$ is actually a Killing vector.

Now we consider the the translation vectors $\partial_i$. First, obviously
$\partial_i$ for $i \neq n$ is obviously a Killing vector field. For $X =
\partial_n$,
\begin{equation}
  \bar{\nabla}_i X^j = X^k \Gamma_{i k}^j = \Gamma_{i n}^j = - (x^n)^{- 1}
  \delta_i^j .
\end{equation}
So
\[ \bar{\nabla}_i X^j + \bar{\nabla}^j X_i = - 2 (x^n)^{- 1} \delta_i^j \]
and $\partial_n$ is a conformal Killing vector. We remark here the recent
article of Jang and Miao {\cite{jang-hyperbolic-2021}} uses $\tfrac{1}{x^n}$
to express the usual mass along horospheres converging to infinity by taking
$\rho (\varepsilon)$ to grow fast. It might be possible to obtain a formula
similar to Theorem \ref{evaluation} evaluating the their mass expression only
along horospheres by exploiting the special properties of the vector $-
\partial_n$.

Now we consider the family $X = X^j \partial_j$ with $X^j$ being
\[ \langle x, a \rangle_{\delta} x^j - \tfrac{1}{2} \langle x, x
   \rangle_{\delta} a^j \]
where $a$ is nonzero constant vector in Euclidean space. We see first that
\begin{equation}
  \partial_i X^j = a_i x^j + \langle x, a \rangle_{\delta} \delta_i^j - x_i
  a^j .
\end{equation}
We consider separately two cases of values of $i, j$. For $i \neq n$, we have
that
\begin{align}
  & \bar{\nabla}_i X^j \\
  = & \partial_i X^j + X^k \Gamma_{i k}^j \\
  = & \partial_i X^j + X^n \Gamma_{i n}^j \\
  = & a_i x^j + \langle x, a \rangle \delta_i^j - x_i a^j - [\langle x, a
  \rangle_{\delta} x^n - \tfrac{1}{2} \langle x, x \rangle_{\delta} a^n]
  (x^n)^{- 1} \delta_i^j \\
  = & a_i x^j - x_i a^j + \tfrac{1}{2} \langle x, x \rangle_{\delta} a^n
  (x^n)^{- 1} \delta_i^j . 
\end{align}
For $i = n$, we have that
\begin{align}
  & \bar{\nabla}_n X^j \\
  = & \partial_n X^j + X^k \Gamma_{n k}^j \\
  = & \partial_n X^j - X^k (x^n)^{- 1} \delta_k^j \\
  = & \partial_n X^j - X^j (x^n)^{- 1} \\
  = & a_n x^j + \langle x, a \rangle \delta_n^j - x_n a^j - [\langle x, a
  \rangle_{\delta} x^j - \tfrac{1}{2} \langle x, x \rangle_{\delta} a^j]
  (x^n)^{- 1} \\
  = & a_n x^j - x_n a^j - \tfrac{1}{x^n} \langle x, a \rangle_{\delta} x^j +
  \tfrac{1}{2 x^n} |x|_{\delta}^2 a^j . 
\end{align}
And for $j \neq n$,
\begin{align}
  & \bar{\nabla}^j X_n = \bar{\nabla}_j X^n \\
  = & \partial_j X^n + X^k \Gamma_{j k}^n \\
  = & \partial_j X^n + X^{\alpha} (x^n)^{- 1} \delta_{\alpha}^j \\
  = & a_j x^n - x_j a^n + \langle x, a \rangle_{\delta} \tfrac{x^j}{x^n} -
  \tfrac{1}{2 x^n} \langle x, x \rangle_{\delta} a^j . 
\end{align}
For $j = n$, we have that
\begin{align}
  & \bar{\nabla}_n X^n \\
  = & \partial_n X^n + X^k \Gamma_{n k}^n \\
  = & \langle x, a \rangle_{\delta} - X^n (x^n)^{- 1} \\
  = & \langle x, a \rangle_{\delta} - [x^n \langle x, a \rangle_{\delta} -
  \tfrac{1}{2} |x|^2_{\delta} a^n] (x^n)^{- 1} \\
  = & \tfrac{1}{2} |x|^2_{\delta} (x^n)^{- 1} a^n . 
\end{align}
To summarize, we have that
\begin{equation}
  \bar{\nabla}_i X^j + \bar{\nabla}^j X_i = \tfrac{1}{2 x^n} \langle x, x
  \rangle_{\delta} a^n \delta_i^j .
\end{equation}
We have that $\langle x, \partial_n \rangle_{\delta} x - \tfrac{1}{2} \langle
x, x \rangle_{\delta} \partial_n$ is a conformal Killing vector and $\langle
x, \partial_i \rangle_{\delta} x - \tfrac{1}{2} \langle x, x \rangle_{\delta}
\partial_i$ is proper Killing vector for all $i \neq n$.

We can construct a conformal Killing vector
\begin{equation}
  \langle x, \partial_n \rangle_{\delta} x - \tfrac{1}{2} \langle x, x
  \rangle_{\delta} \partial_n - x
\end{equation}
which has no tangential component to $\mathcal{H}$.

Now we consider the the rotation vectors $x^i \partial_j - x^j \partial_i$. If
both $i$ and $j$ is not $n$, we see easily that $x^i \partial_j - x^j
\partial_i$ is a Killing vector.

We consider now the vector $X = x^n \partial_k - x^k \partial_n$ with $k \neq
n$, the components $X^j = x^n \delta_k^j - x^k \delta_n^j$. We compute
$\bar{\nabla}_i X^j$. For $i \neq n, j \neq n$,
\begin{align}
  & \bar{\nabla}_i X^j \\
  = & \partial_i X^j + X^k \Gamma_{i k}^j \\
  = & \delta_i^n \delta_j^k - \delta_i^k \delta_n^j + X^n \Gamma_{i n}^j
  \\
  = & \delta_i^n \delta_j^k - \delta_i^k \delta_n^j + x^k (x^n)^{- 1}
  \delta_i^j \\
  = & x^k (x^n)^{- 1} \delta_i^j . 
\end{align}
So
\begin{equation}
  \bar{\nabla}_i X^j + \bar{\nabla}_j X^i = X^k \Gamma_{i k}^j + X^k \Gamma_{j
  k}^i = 2 \tfrac{x^k}{x^n} \delta_i^j .
\end{equation}
We have for $j \neq n$,
\begin{align}
  & \bar{\nabla}_n X^j \\
  = & \partial_n X^j + X^l \Gamma_{n l}^j \\
  = & \delta_k^j - \delta^k_n \delta_n^j + X^{\alpha} \Gamma_{n \alpha}^j
  \\
  = & \delta_k^j - \delta_n^k \delta_n^j + x^n \Gamma^j_{n k} \\
  = & \delta_k^j - \delta_n^k \delta_n^j - \delta^j_k = 0. 
\end{align}
Similarly, $\bar{\nabla}_j X^n = 0$. And
\begin{equation}
  \bar{\nabla}_n X^n = X^j \Gamma^n_{j n} = X^n \Gamma_{n n}^n =
  \tfrac{x^k}{x^n} .
\end{equation}
So we see that
\begin{equation}
  \bar{\nabla}_i X^j + \bar{\nabla}^j X_i = \tfrac{2 x^k}{x^n} \delta_i^j
\end{equation}
for all $i$ and $j$.

We consider the vectors $Y = x - \partial_n$ and
\begin{equation}
  Y^{(k)} = (x^n \partial_k - x^k \partial_n) + (\langle x, \partial_k \rangle
  x - \tfrac{1}{2} \langle x, x \rangle_{\delta} \partial_k) .
\end{equation}
Both $Y$ and $Y^{(k)}$ are tangent to $\mathcal{H}$ since $x^n = 1$.

The construction is by shifting a conformal Killing vector by a Killing
vector. This is motivated by a recent work of the author
{\cite{chai-note-2021}}. We have
\begin{equation}
  \ensuremath{\operatorname{div}}Y = \tfrac{n}{x^n}, \:
  \ensuremath{\operatorname{div}}Y^{(k)} = n \tfrac{x^k}{x^n} \label{vector
  and potential}
\end{equation}
by previous calculations. Along $\mathcal{H}$, $Y = \hat{x}$ and
\[ Y^{(k)} = \tfrac{1}{2} \partial_k + \langle \hat{x}, \partial_k
   \rangle_{\delta} - \tfrac{1}{2} \langle \hat{x}, \hat{x} \rangle_{\delta}
   \partial_k . \]
They are conformal Killing vectors along $\mathcal{H}$, specifically
\begin{equation}
  \partial_{\alpha} Y^{\beta} + \partial^{\beta} Y_{\alpha} = 2
  \delta_{\alpha}^{\beta}, \: \ensuremath{\operatorname{div}}_{\mathcal{H}} Y
  = n - 1
\end{equation}
and
\begin{equation}
  \partial_{\alpha} (Y^{(k)})^{\beta} + \partial^{\beta} (Y^{(k)})_{\alpha} =
  2 \langle \hat{x}, \partial_k \rangle_{\delta} \delta_{\alpha}^{\beta}, \:
  \ensuremath{\operatorname{div}}_{\mathcal{H}} Y^{(k)} = (n - 1) \langle
  \widehat{x,} \partial_k \rangle_{\delta} . \label{vector and potential
  horosphere}
\end{equation}
\begin{remark}
  The vectors $Y$ and $Y^{(k)}$ constructed here can be used to prove similar
  results as in {\cite{wang-uniqueness-2019}} by considering free boundary
  hypersurfaces supported on the horosphere.
\end{remark}

\begin{lemma}
  These vectors admit the growth rate
  \begin{equation}
    |Y|_b + | \bar{\nabla} Y|_b + |Y^{(k)} |_b + | \bar{\nabla} Y^{(k)} |_b =
    O (\cosh r) \label{vector lenght grwoth}
  \end{equation}
  as $r \to \infty$.
\end{lemma}

\begin{proof}
  The proof is by direct calculation for each term. The length of $Y = x -
  \partial_n$ is
  \begin{align}
    |Y|_b = & \tfrac{1}{x^n} \sqrt{| \hat{x} |^2 + (x^n - 1)^2} \\
    \leqslant & \tfrac{\sqrt{2}}{x^n} (| \hat{x} | + |x^n - 1|) \\
    \leqslant & \tfrac{\sqrt{2}}{x^n} (| \hat{x} |^2 + x^n + 1) \\
    = & O (\cosh r), 
\end{align}
  according to {\eqref{distance formula}}. For $i \neq n$, $j \neq n$,
  $\bar{\nabla}_i Y^j = \tfrac{\delta_i^j}{x^n}$; $\bar{\nabla}_i Y^n = -
  \bar{\nabla}_n Y^i = - \tfrac{x^i}{x^n}$; $\bar{\nabla}_n Y^n =
  \tfrac{1}{x^n}$. So the length of $\bar{\nabla} Y$ is
  \begin{equation}
    | \bar{\nabla} Y|_b = \tfrac{1}{x^n} (n + 2| \hat{x} |^2)^{\tfrac{1}{2}} =
    O (\tfrac{| \hat{x} |}{x^n}) = O (\cosh r) .
  \end{equation}
  The length of $Y^{(k)}$ is
  \begin{align}
    & |Y|^2_b \\
    = & (x^n)^{- 2} [(x^n)^2 + (x^k)^2 + \tfrac{1}{4} \langle x, x
    \rangle_{\delta}^2 - \langle x, x \rangle_{\delta} x^n] \\
    = & \tfrac{1}{(x^n)^2} [(x^k)^2 + (x^n - \tfrac{1}{2} \langle x, x
    \rangle_{\delta})^2] \\
    \leqslant & \tfrac{1}{(x^n)^2} [| \hat{x} |^2 + 2 (x^n)^2 + \tfrac{1}{2}
    \langle x, x \rangle_{\delta}^2] \\
    \leqslant & \tfrac{1}{(x^n)^2} [| \hat{x} |^2 + 2 (x^n)^2 + | \hat{x} |^4
    + (x^n)^4] . 
\end{align}
  We see then that $|Y|_b = O (\cosh r)$. We write $U = Y^{(k)}$, we have that
  for $i \neq n$, $j \neq n$,
  \begin{equation}
    \bar{\nabla}_i U^j = a_i x^j - x_i a^j + \tfrac{x^k}{x^n} \delta_i^j .
  \end{equation}
  For $j \neq n$,
  \begin{equation}
    \bar{\nabla}_j U^n = - \bar{\nabla}_n U^j = a_j x^n + \tfrac{x^j x^k}{x^n}
    - \tfrac{1}{2 x^n} \langle x, x \rangle_{\delta} a^j
  \end{equation}
  and $\bar{\nabla}_n U^n = \tfrac{x^k}{x^n}$. Here $a$ is the vector
  $\partial_k$. Summing up these entries of $\bar{\nabla} U$, we have that
  \begin{align}
    & | \bar{\nabla} U|_b^2 \\
    = & \tfrac{(x^k)^2}{(x^n)^2} + \sum_{i \neq n, j \neq n} (a_i x^j - x_i
    a^j + \tfrac{x^k}{x^n} \delta_i^j)^2 \\
    & + 2 \sum_{j \neq n} (a_j x^n + \tfrac{x^j x^k}{x^n} - \tfrac{1}{2 x^n}
    \langle x, x \rangle_{\delta} a^j)^2 \\
    = & \tfrac{(x^k)^2}{(x^n)^2} + [| \hat{x} |^2 + | \hat{x} |^2 + (n - 1)
    \tfrac{(x^k)^2}{(x^n)^2} - 2 (x^k)^2] \\
    & + 2 [(x^n)^2 + \tfrac{| \hat{x} |^2 (x^k)^2}{(x^n)^2} + \tfrac{\langle
    x, x \rangle_{\delta}^2}{4 (x^n)^2} + 2 (x^k)^2 - \langle x, x
    \rangle_{\delta} - \tfrac{1}{(x^n)^2} (x^k)^2 \langle x, x
    \rangle_{\delta}] . 
\end{align}
  Using the relation that $\langle x, x \rangle_{\delta}^2 = (x^n)^2 + |
  \hat{x} |^2$, the last line reduces to
  \begin{equation}
    | \bar{\nabla} U|_b^2 = (x^n)^{- 2} (\tfrac{1}{2} \langle x, x
    \rangle_{\delta}^2 + (x^k)^2) .
  \end{equation}
  We see then $| \bar{\nabla} U|_b = O (\cosh r)$.
\end{proof}

\section{Proof}\label{proof}

\subsection{Finiteness of $\mathbf{M} (V)$}

First, we show that the quantity $\mathbf{M} (V)$ is well defined under
natural conditions.

\begin{theorem}
  \label{finiteness}If $(M, g)$ is an asymptotically hyperbolic manifold with
  a horospherical boundary and if $\mathrm{e}^r (R_g + n (n - 1)) \in L^1 (M)$
  and $\mathrm{e}^r (H + n - 1) \in L^1 (M)$, then the quantity $\mathbf{M}
  (V)$ defined in {\eqref{mass}} exists and is finite.
\end{theorem}

We are concerned only with behavior near infinity, so we can assume that $M$
is diffeomorphic to $\mathbb{H}^n$ minus a horoball. We then set up notations.

Before going to the proof of Theorem \ref{finiteness}, we have the following
elementary lemma concerning the integral of $\cosh^{- 2 \tau + 1} r$ on the
regions defined as
\begin{align}
  I_1 & = \{\varepsilon_1 \leqslant x^n \leqslant 1\} \cap \{\rho
  (\varepsilon_1) \leqslant | \hat{x} | \leqslant \rho (\varepsilon_2)\},
  \\
  I_2 & = \{\varepsilon_2 \leqslant x^n \leqslant \varepsilon_1 \} \cap \{|
  \hat{x} | \leqslant \rho (\varepsilon_2)\}, \\
  I_3 & = \{1 \leqslant x^n \leqslant \varepsilon^{- 1}_1 \} \cap \{\rho
  (\varepsilon_1) \leqslant | \hat{x} | \leqslant \rho (\varepsilon_2)\},
  \\
  I_4 & = \{\varepsilon_1^{- 1} \leqslant x^n \leqslant \varepsilon_2^{- 1} \}
  \cap \{| \hat{x} | \leqslant \rho (\varepsilon_2)\} . 
\end{align}
\begin{lemma}
  \label{small terms in annular region}Assume that $\rho (\varepsilon) \to
  \infty$ as $\varepsilon \to 0$, we have that for $k \in \{1, 2, 3, 4\}$,
  \begin{equation}
    \int_{I_k} \cosh^{- 2 \tau + 1} r \mathrm{d} \bar{v} \to 0 \label{only one
    kind small term}
  \end{equation}
  if $\varepsilon_1 \to 0$ and $\varepsilon_2 \to 0$.
\end{lemma}

\begin{proof}
  We deal with $I_1$ first. From {\eqref{distance formula}},
  \begin{align}
    & \int_{I_1} \cosh^{- 2 \tau + 1} r \mathrm{d} \bar{v} \\
    = & \int_{\varepsilon_1}^1 \mathrm{d} x^n \int_{\rho (\varepsilon_1)
    \leqslant | \hat{x} | \leqslant \rho (\varepsilon_2)} (x^n)^{- n + 1}
    \cosh^{- 2 \tau + 1} r \mathrm{d} \hat{x} \\
    = & \int^1_{\varepsilon_1} (x^n)^{- n + 1} \mathrm{d} x^n \int_{\rho
    (\varepsilon_1) \leqslant | \hat{x} | \leqslant \rho (\varepsilon_2)}
    (\frac{| \hat{x} |^2 + (x^n)^2 + 1}{2 x^n})^{- 2 \tau + 1} \mathrm{d}
    \hat{x} \\
    \leqslant & 2^{1 - 2 \tau} \int^1_{\varepsilon_1} (x^n)^{- n + 2 \tau}
    \mathrm{d} x^n \int_{\rho (\varepsilon_1) \leqslant | \hat{x} | \leqslant
    \rho (\varepsilon_2)} | \hat{x} |^{2 - 4 \tau} \mathrm{d} \hat{x}
    \\
    \leqslant & C \int_{\rho (\varepsilon_1)}^{\rho (\varepsilon_2)} s^{2 - 4
    \tau + n - 2} \mathrm{d} s 
\end{align}
  Note that $- n + 2 \tau > 0$ and $2 - 4 \tau + n - 2 < - n < - 1$, this is
  $o (1)$ obviously as long as as $\varepsilon \to 0$, $\rho (\varepsilon)
  \nearrow \infty$. Similarly, on the region $I_2$,
  \begin{align}
    & \int_{I_2} \cosh^{- 2 \tau + 1} r \mathrm{d} \bar{v} \\
    \leqslant & 2^{1 - 2 \tau} \int_{\varepsilon_2}^{\varepsilon_1} (x^n)^{- n
    + 2 \tau} \mathrm{d} x^n \int_0^{\rho (\varepsilon_2)} (\rho^2 + 1)^{- 2
    \tau + 1} \rho^{n - 2} \mathrm{d} s 
\end{align}
  which is also $o (1)$ as $\varepsilon_1 \to 0$ and $\varepsilon_2 \to 0$.
  
  For the integral over $I_3$,
  \begin{equation}
    \int_1^{\varepsilon_1^{- 1}} \mathrm{d} x^n \int_{\rho_1 \leqslant |
    \hat{x} | \leqslant \rho_2} (x^n)^{- n + 1} \cosh^{- 2 \tau + 1} r
    \mathrm{d} \hat{x}
  \end{equation}
  We have that
  \begin{align}
    & \int_1^{\varepsilon_1^{- 1}} \int_{\rho_1 \leqslant | \hat{x} |
    \leqslant \rho_2} \left( \tfrac{| \hat{x} |^2 + (x^n)^2 + 1}{2 x^n}
    \right)^{- 2 \tau + 1} (x^n)^{- n + 1} \mathrm{d} \hat{x} \mathrm{d} x^n
    \\
    = & 2^{1 - 2 \tau} \int_1^{\varepsilon_1^{- 1}} \int_{\rho
    (\varepsilon_1)}^{\rho (\varepsilon_2)} (x^n)^{- n + 2 \tau} (s^2 +
    (x^n)^2 + 1)^{- 2 \tau + 1} \mathrm{d} s \mathrm{d} x^n \\
    \leqslant & C \int_1^{\varepsilon_1^{- 1}} \int_{\rho
    (\varepsilon_1)}^{\rho (\varepsilon_2)} (x^n)^{- n + 2 \tau} (s^2 +
    (x^n)^2 + 1)^{\tfrac{n - 2 \tau - p}{2} + (- 2 \tau + 1 - \tfrac{n - 2
    \tau - p}{2})} s^{n - 2} \mathrm{d} s \mathrm{d} x^n \\
    \leqslant & C \int_1^{\varepsilon_1^{- 1}} (x^n)^{- p} \mathrm{d} x^n
    \int_{\rho (\varepsilon_1)}^{\rho (\varepsilon_2)} (s^2 + 1)^{- 2 \tau + 1
    - \tfrac{n - 2 \tau - p}{2}} s^{n - 2} \mathrm{d} s \\
    \leqslant & C \int_1^{\infty} t^{- p} \mathrm{d} t \int_{\rho
    (\varepsilon_1)}^{\infty} (s^2 + 1)^{- 2 \tau + 1 - \tfrac{n - 2 \tau -
    p}{2}} s^{n - 2} \mathrm{d} s. 
\end{align}
  We fix some $p$ with $1 < p < 2 \tau - 1$. Then power
  \[ n - 2 + 2 (- 2 \tau + 1 - \tfrac{n - 2 \tau - p}{2}) \]
  is less than $- 1$. We see the integral is $o (1)$ as $\varepsilon_1 \to 0$.
  With a similar argument, on the region $I_4$,
  \begin{align}
    & \int_{\varepsilon_1^{- 1}}^{\varepsilon_2^{- 1}} \int_{| \hat{x} |
    \leqslant \rho_2} (x^n)^{- n + 1} \cosh^{- 2 \tau + 1} r \mathrm{d}
    \hat{x} \mathrm{d} x^n \\
    \leqslant & C \int_{\varepsilon_1^{- 1}}^{\varepsilon_2^{- 1}} t^{- p}
    \mathrm{d} t \int_0^{\rho (\varepsilon_2)} (s^2 + 1)^{- 2 \tau + 1 -
    \tfrac{n - 2 \tau - p}{2}} s^{n - 2} \mathrm{d} s \\
    \leqslant & C \int_{\varepsilon_1^{- 1}}^{\infty} t^{- p} \mathrm{d} t
    \int_0^{\infty} (s^2 + 1)^{- 2 \tau + 1 - \tfrac{n - 2 \tau - p}{2}} s^{n
    - 2} \mathrm{d} s 
\end{align}
  we fix the same $p$ as before, this integral is also $o (1)$.
\end{proof}

\begin{proof}[Proof of Theorem \ref{finiteness}]
  The proof is basically a restatement of the expansion we derived earlier.
  See {\cite{chrusciel-mass-2003}}. We have the expansion of the scalar
  curvature $R_g$ near $b$ that
  \begin{equation}
    R_g = - n (n - 1) + D R (e) + O (\mathrm{e}^{- 2 \tau r}) .
  \end{equation}
  The specific form of $O (e^{- 2 \tau r})$ is $O (|e|^2 + | \bar{\nabla} e|^2
  + |e|| \bar{\nabla}^2 e|)$ (See for example {\cite{chai-evaluation-2018}}).
  Here,
  \begin{equation}
    D R (e) =\ensuremath{\operatorname{div}} (\ensuremath{\operatorname{div}}e
    - \mathrm{d} E) + (n - 1) E
  \end{equation}
  is the linearization operator of the scalar curvature. We have
  \begin{equation}
    V D R (e) = \langle D^{\ast} R (V), e \rangle + \bar{\nabla}_i
    \mathbb{U}^i
  \end{equation}
  where $D^{\ast} R = \bar{\nabla}^2 V - V b$ is the formal $L^2$ adjoint of
  $D R$. Since $V$ is the static potential {\eqref{static potential full}}, we
  have
  \begin{equation}
    V (R_g + n (n - 1)) = \bar{\nabla}_i \mathbb{U}^i + O (\mathrm{e}^{- 2
    \tau r + r}) . \label{expansion of scalar curvature near infinity}
  \end{equation}
  Now we integrate {\eqref{expansion of scalar curvature near infinity}} over
  the region $A = A_{\varepsilon_2, \varepsilon_1}$, we see that
  \begin{align}
    & \int_A V (R_g + n (n - 1)) \\
    = & \int_{\partial' C (\varepsilon_2)} \mathbb{U}^i \bar{\eta}_i -
    \int_{\partial' C (\varepsilon_1)} \mathbb{U}^i \bar{\eta}_i + \int_A O
    (\mathrm{e}^{- 2 \tau r + r}) + \int_{a_{\varepsilon_2, \varepsilon_1}}
    \mathbb{U}^i \bar{\eta}_i . 
\end{align}
  Using {\eqref{divergence horosphere}}, so
  \begin{align}
    & \int_A V (R_g + n (n - 1)) \\
    = & \int_{\partial' C (\varepsilon_2)} \mathbb{U}^i \bar{\eta}_i -
    \int_{\partial' C (\varepsilon_1)} \mathbb{U}^i \bar{\eta}_i + \int_A O
    (\mathrm{e}^{- 2 \tau r + r}) \\
    & + \int_{a_{\varepsilon_2, \varepsilon_1}} [- \partial_{\alpha} (V
    e_{\alpha n}) - 2 V (H_g + n - 1)] . 
\end{align}
  Using divergence theorem on $a_{\varepsilon_2, \varepsilon_1}$, we have that
  \begin{align}
    & (\int_{\partial' C (\varepsilon_2)} \mathbb{U}^i \bar{\eta}_i -
    \int_{\partial c_{\varepsilon_2}} V e_{\alpha n} \theta^{\alpha}) -
    (\int_{\partial' C (\varepsilon_2)} \mathbb{U}^i \bar{\eta}_i -
    \int_{\partial c_{\varepsilon_1}} V e_{\alpha n} \theta^{\alpha})
    \\
    = & \int_A V (R_g + n (n - 1)) + 2 \int_{a_{\varepsilon_2, \varepsilon_1}}
    V (H_g + n - 1) + \int_A O (\mathrm{e}^{- 2 \tau r + r}) \\
    = & \int_A O (\mathrm{e}^{- 2 \tau r + r}) + \int_{a_{\varepsilon_2,
    \varepsilon_1}} O (\mathrm{e}^{- 2 \tau r + r}) = o (1) \label{mass well
    defined} 
\end{align}
  by Lemma \ref{small terms in annular region}, {\eqref{distance formula}},
  and the integrability of $V (R_g + n (n - 1))$ and $V (H_g + n - 1)$.
  Therefore, we have shown that $\mathbf{M} (V)$ exists and is finite.
\end{proof}

Now we turn to the proof of Theorem \ref{evaluation}.

\begin{proof}[Proof of Theorem \ref{evaluation}]
  We use the method of {\cite{herzlich-computing-2016}}. We define a cutoff
  function $\chi$ which vanish inside $C (\varepsilon^{\tfrac{1}{2}})$, equals
  1 outside $C (\varepsilon^{\tfrac{3}{4}})$ and
  \[ \hat{g} = \chi g + (1 - \chi) b \]
  The cutoff function is a product of two cutoff functions $\chi = \chi_1
  (x^n) \chi_2 (| \hat{x} |)$. The function $\chi_1 (x^n) = f (- \log x^n)$
  where $f (t)$ is the cutoff vanish inside $[0, - \tfrac{1}{2} \log
  \varepsilon]$ and equal to 1 in $[- \tfrac{3}{4} \log \varepsilon, \infty)$
  with the estimate
  \begin{equation}
    f + (\log \varepsilon) |f' | + (\log \varepsilon)^2 |f'' | \leqslant C.
  \end{equation}
  We find then
  \[ | \nabla \chi_1 |_b = x^n | \tfrac{\partial \chi_1}{\partial x^n} | = x^n
     \cdot \tfrac{1}{x^n} f' (- \log x^n) \leqslant \tfrac{C}{\log
     \varepsilon} \leqslant C \]
  and similarly $| \nabla^2 \chi_1 |_b \leqslant \tfrac{C}{(\log
  \varepsilon)^2} \leqslant C$. We define $\chi_2 (t)$ to be the standard
  cutoff which vanishes inside $[0, \varepsilon^{- \tfrac{1}{2} \alpha}]$ and
  is equal to 1 in $[\varepsilon^{- \tfrac{3}{4} \alpha}, \infty)$ with the
  estimate
  \begin{equation}
    0 \leqslant \chi_2 \leqslant 1, | \chi_2' | \leqslant C (\varepsilon^{-
    \tfrac{3}{4} \alpha} - \varepsilon^{- \tfrac{1}{2} \alpha})^{- 1}, |
    \chi_2' |^2 \leqslant C (\varepsilon^{- \tfrac{3}{4} \alpha} -
    \varepsilon^{- \tfrac{1}{2} \alpha})^{- 2} .
  \end{equation}
  It is easy to check that
  \begin{equation}
    | \chi_2 | + | \bar{\nabla} \chi_2 | + | \bar{\nabla}^2 \chi_2 | \leqslant
    C
  \end{equation}
  when $x^n \in (0, \varepsilon^{- 1}]$ due to the condition $\alpha >
  \tfrac{4}{3}$. From the construction of the cutoff function, $\hat{g}$ is an
  asymptotically hyperbolic metric, that is the difference $\hat{e} = \hat{g}
  - b$ satisfies also the decay rate {\eqref{decay rate}}.
  
  The reader might want to visit the the model $\mathrm{e}^{2 t} | \mathrm{d}
  \hat{x} |^2 + \mathrm{d} t^2$ where $t = - \log x^n$ for a better
  understanding of the construction of the cutoff function. We use mostly the
  upper half space model because of the conformality used in calculation of
  conformal Killing vectors.
  
  We shall also denote by $\hat{g}$ the complete metric obtained by gluing the
  hyperbolic metric inside $C (\varepsilon^{\tfrac{1}{4}})$ and the metric $g$
  outside $C_{\varepsilon}$.
  
  By divergence theorem, we have on $A = C_{\varepsilon} \backslash C
  (\varepsilon^{\tfrac{1}{4}})$ by the divergence theorem and the second
  Bianchi identity that
  \begin{equation}
    \int_{\partial A} \hat{G} (X, \hat{\nu}) \mathrm{d} \hat{\sigma} = \int_A
    \hat{\nabla}^i (\hat{G}_{i j} X^j) \mathrm{d} \hat{v} = \int_A \hat{G}_{i
    j} \hat{\nabla}^i X^j \mathrm{d} \hat{v} . \label{divergence bianchi}
  \end{equation}
  We first analyze the term $\int_A \hat{G}_{i j} \hat{\nabla}^i X^j
  \mathrm{d} \hat{v}$. We use
  \begin{align}
    \hat{\nabla}^i X^j = & h^{i l} (\bar{\nabla}_l X^j + X^k (\hat{\Gamma}_{i
    k}^l - \bar{\Gamma}_{i k}^l)) \\
    = & \bar{\nabla}^i X^j + (\hat{g}^{i l} - b^{i l}) \bar{\nabla}_l X^j +
    h^{i l} X^k (\hat{\Gamma}_{i k}^l - \bar{\Gamma}_{i k}^l) . 
\end{align}
  In fact $\int_A \hat{G}_{i j} (\hat{g}^{i l} - b^{i l}) \bar{\nabla}_l X^j
  \mathrm{d} \hat{v}$ is $o (1)$ by noting that the decay of $\hat{G}^i_j
  (\hat{g}^{i l} - b^{i l}) \bar{\nabla}_l X^j = O (\cosh^{- 2 \tau} r|
  \bar{\nabla} X|)$ and the growth rate {\eqref{vector lenght grwoth}}.
  Similarly $\int_A \hat{G}_{i j} \hat{g}^{i l} X^k (\hat{\Gamma}_{i k}^l -
  \bar{\Gamma}_{i k}^l) \mathrm{d} \hat{v}$ is $o (1)$.
  \begin{align}
    & \int_A \hat{G}_{i j} \bar{\nabla}^i X^j \mathrm{d} \hat{v} \\
    = & \tfrac{1}{2} \int_{\mathcal{A}} \hat{G}_{i j} (\bar{\nabla}^i X^j +
    \bar{\nabla}^j X^i) \mathrm{d} \hat{v} \\
    = & \tfrac{1}{n} \int_A \ensuremath{\operatorname{div}}^b X \hat{G}_{i j}
    b^{i j} \mathrm{d} \hat{v} \\
    = & \tfrac{1}{n} \int_A \ensuremath{\operatorname{div}}^b X \hat{G}_{i j}
    \hat{g}^{i j} \mathrm{d} \hat{v} - \tfrac{1}{n} \int_A
    \ensuremath{\operatorname{div}}^b X \hat{G}_{i j} (\hat{g}^{i j} - b^{i
    j}) \mathrm{d} \hat{v} \\
    = & \tfrac{2 - n}{2 n} \int_A \ensuremath{\operatorname{div}}^b X (\hat{R}
    + n (n - 1)) \mathrm{d} \hat{v} - \tfrac{1}{n} \int_A
    \ensuremath{\operatorname{div}}^b X \hat{G}_{i j} (h^{i j} - b^{i j})
    \mathrm{d} \hat{v} 
\end{align}
  We have that similarly the two terms $\int_A
  \ensuremath{\operatorname{div}}^b X \hat{G}_{i j} (h^{i j} - b^{i j})
  \mathrm{d} \hat{v}$ and $\int_A \ensuremath{\operatorname{div}}^b X (\hat{R}
  + n (n - 1)) (\mathrm{d} \hat{v} - \mathrm{d} \bar{v})$ are $o (1)$.
  Therefore,
  \begin{equation}
    \int_A \hat{G}_{i j} \hat{\nabla}^i X^j \mathrm{d} \hat{v} = \tfrac{2 -
    n}{2 n} \int_A \ensuremath{\operatorname{div}}^b X (\hat{R} + n (n - 1))
    \mathrm{d} \bar{v} + o (1) . \label{interior brown york}
  \end{equation}
  We are concerned about $\int_a \hat{G} (X, \hat{\nu}) \mathrm{d}
  \hat{\sigma}$. Denote $a = c_{\varepsilon} \backslash c
  (\varepsilon^{\tfrac{1}{4}})$. Since $X$ is tangent to the $\mathcal{H}$, we
  use Gauss-Codazzi equation, we find that
  \begin{equation}
    \int_a \hat{G} (X, \nu) \mathrm{d} \hat{\sigma} = \int_a X^{\beta}
    \hat{D}^{\alpha} (\hat{A}_{\alpha \beta} - \hat{H} \hat{\rho}_{\alpha
    \beta}) \mathrm{d} \hat{\sigma} . \label{gauss codazzi on horosphere}
  \end{equation}
  We add an extra term to the tensor $\hat{A} - \hat{H} \hat{\rho}$ so that it
  is trace free if $g$ is just the hyperbolic metric $b$. The modification is
  in spirit similar to the modified Einstein tensor $G$. So

  \begin{equation}
    \int_a \hat{G} (X, \hat{\nu}) \mathrm{d} \hat{\sigma} = \int_a X^{\beta}
    \hat{D}^{\alpha} (\hat{A}_{\alpha \beta} - \hat{H} \hat{\rho}_{\alpha
    \beta} - (n - 2) h_{\alpha \beta}) \mathrm{d} \hat{\sigma} .
  \end{equation}
  We apply the divergence theorem, we obtain
  \begin{equation}
    \int_a \hat{G} (X, \hat{\nu}) \mathrm{d} \hat{\sigma} = \int_{\partial a}
    \hat{W} (X, \hat{\mu}) \mathrm{d} \hat{\lambda} - \int_a \hat{W}_{\alpha
    \beta} \hat{D}^{\alpha} X^{\beta} \mathrm{d} \hat{\sigma} .
    \label{divergence gauss codazzi}
  \end{equation}
  Now we analyze the term $\int_a \hat{W}_{\alpha \beta} \hat{D}^{\alpha}
  X^{\beta} \mathrm{d} \hat{\sigma}$. We use
  \begin{align}
    \hat{D}^{\alpha} X^{\beta} = & \hat{\rho}^{\alpha \gamma}
    (\bar{D}_{\gamma} X^{\beta} + X^{\xi} (\hat{\Lambda}_{\gamma \xi}^{\beta}
    - \bar{\Lambda}_{\gamma \xi}^{\beta})) \\
    = & \bar{D}^{\alpha} X^{\beta} + (\hat{\rho}^{\alpha \gamma} -
    \bar{\rho}^{\alpha \gamma}) \bar{D}_{\gamma} X^{\beta} +
    \hat{\rho}^{\alpha \gamma} X^{\xi} (\hat{\Lambda}_{\gamma \xi}^{\beta} -
    \bar{\Lambda}_{\gamma \xi}^{\beta}) . 
\end{align}
  The terms $\int_a \hat{W}_{\alpha \beta} (\hat{\rho}^{\alpha \gamma} -
  \bar{\rho}^{\alpha \gamma}) \bar{D}_{\gamma} X^{\beta} \mathrm{d}
  \hat{\sigma}$ and $\int_a \hat{\rho}^{\alpha \gamma} X^{\xi}
  (\hat{\Lambda}_{\gamma \xi}^{\beta} - \bar{\Lambda}_{\gamma \xi}^{\beta})
  \mathrm{d} \hat{\sigma}$ are $o (1)$.
  
  We know that
  \begin{align}
    & \int_a \hat{W}_{\alpha \beta} \bar{D}^{\alpha} X^{\beta} \mathrm{d}
    \hat{\sigma} \\
    = & \tfrac{1}{2} \int_a \hat{W}_{\alpha \beta} (\bar{D}^{\alpha} X^{\beta}
    + \bar{D}^{\beta} X^{\alpha}) \mathrm{d} \hat{\sigma} \\
    = & \tfrac{1}{n - 1} \int_a
    \ensuremath{\operatorname{div}}^b_{\mathcal{H}} X \bar{\rho}^{\alpha
    \beta} \hat{W}_{\alpha \beta} \mathrm{d} \hat{\sigma} \\
    = & \tfrac{2 - n}{n - 1} \int_a
    \ensuremath{\operatorname{div}}^b_{\mathcal{H}} X \hat{\rho}^{\alpha
    \beta} \hat{W}_{\alpha \beta} \mathrm{d} \hat{\sigma} - \tfrac{1}{n - 1}
    \int_a \ensuremath{\operatorname{div}}^b_{\mathcal{H}} X
    (\hat{\rho}^{\alpha \beta} - \bar{\rho}^{\alpha \beta}) \hat{W}_{\alpha
    \beta} \mathrm{d} \hat{\sigma} 
\end{align}
  We have that $\int_a \ensuremath{\operatorname{div}}^b_{\mathcal{H}} X
  (\hat{\rho}^{\alpha \beta} - \bar{\rho}^{\alpha \beta}) \hat{W}_{\alpha
  \beta} \mathrm{d} \hat{\sigma}$ and $\int_a
  \ensuremath{\operatorname{div}}^b_{\mathcal{H}} X \hat{\rho}^{\alpha \beta}
  \hat{W}_{\alpha \beta} (\mathrm{d} \hat{\sigma} - \mathrm{d} \bar{\sigma})$
  are $o (1)$. Therefore,
  \begin{equation}
    \int_a \hat{W}_{\alpha \beta} \hat{D}^{\alpha} X^{\beta} \mathrm{d}
    \hat{\sigma} = \tfrac{2 - n}{n - 1} \int_a
    \ensuremath{\operatorname{div}}^b_{\mathcal{H}} X (\hat{H} + n - 1)
    \mathrm{d} \bar{\sigma} + o (1) . \label{boundary brown york}
  \end{equation}
  We have from {\eqref{divergence bianchi}}, {\eqref{interior brown york}} ,
  {\eqref{divergence gauss codazzi}} and {\eqref{boundary brown york}} that
  \begin{align}
    & \int_{\partial A} \hat{G} (X, \hat{\nu}) \mathrm{d} \hat{\sigma}
    \\
    = & \int_{(\partial A) \backslash a} \hat{G} (X, \hat{\nu}) \mathrm{d}
    \hat{\sigma} + \int_a \hat{G} (X, \hat{\nu}) \mathrm{d} \hat{\sigma}
    \\
    = & \int_{(\partial A) \backslash a} \hat{G} (X, \hat{\nu}) \mathrm{d}
    \hat{\sigma} + \int_{\partial a} \hat{W} (X, \hat{\mu}) \mathrm{d}
    \hat{\lambda} - \tfrac{2 - n}{n - 1} \int_a
    \ensuremath{\operatorname{div}}^b_{\mathcal{H}} X (\hat{H} + n - 1)
    \mathrm{d} \bar{\sigma} \\
    = & \tfrac{2 - n}{2 n} \int_A \ensuremath{\operatorname{div}}^b X (\hat{R}
    + n (n - 1)) \mathrm{d} \bar{v} + o (1) . 
\end{align}
  Since $\ensuremath{\operatorname{div}}^b X = n V$ and
  $\ensuremath{\operatorname{div}}_{\mathcal{H}}^b X = (n - 1) V$ along
  $\mathcal{H}$, by the same argument arriving {\eqref{mass well defined}}, we
  have
  \begin{align}
    & \tfrac{1}{2} (2 - n) \mathbf{M} (V) + o (1) \\
    = & \tfrac{2 - n}{2 n} \int_A \ensuremath{\operatorname{div}}^b X (\hat{R}
    + n (n - 1)) \mathrm{d} \bar{v} + \tfrac{2 - n}{n - 1} \int_a
    \ensuremath{\operatorname{div}}^b_{\mathcal{H}} X (\hat{H} + n - 1)
    \mathrm{d} \bar{\sigma} \\
    = & \int_{(\partial A) \backslash a} \hat{G} (X, \hat{\nu}) \mathrm{d}
    \hat{\sigma} + \int_{\partial a} \hat{W} (X, \hat{\mu}) \mathrm{d}
    \hat{\lambda} . 
\end{align}
  Now since the metric $\hat{g}$ is standard inside $C
  (\varepsilon^{\tfrac{1}{2}})$ and $\hat{g} = g$ outside $C
  (\varepsilon^{\tfrac{3}{4}})$, we have the desired formula.
\end{proof}

\begin{remark}
  It might be possible to drop the condition that $\rho (\varepsilon) =
  \varepsilon^{- \alpha}$ where $\alpha > \tfrac{4}{3}$ using the expansion
  and the method developed by the author {\cite{chai-evaluation-2018}}.
\end{remark}

\

\

\end{document}